\documentclass[11pt,a4paper,reqno, final]{article}

\usepackage{template}
\usepackage{bbm, fullpage}
\usepackage[notcite,notref]{showkeys}
\usepackage{titlesec}
\titleformat{\subsection}[runin]
        {\normalfont\bfseries}
        {\thesubsection}
        {0.5em}
        {}
        []

\renewcommand{\P}{\mathbb{P}}
\newcommand{\E}{\mathbb{E}}
\renewcommand{\d}{{\rm d}}
\newcommand{\pnorm}[2]{\left\|#2\right\|_{#1}}

\newcommand{\Z}{\mathbb{Z}}
\newcommand{\R}{\mathbb{R}}
\newcommand{\N}{\mathbb{N}}

\newcommand{\bb}[1]{\mathbb{#1}}
\newcommand{\ignore}[1]{}

\def\sG{\mathcal{G}}

\def\sT{\mathcal{T}}

\begin{document}

\title{
Marked random graphs with given degree sequence:\\ large deviations on the local topology and applications
}

\date{\today}
\author{Rangel Baldasso
  \thanks{E-mail: \ rangel@puc-rio.br; \ Department of Mathematics, PUC-Rio, Rua Marqu\^{e}s de S\~{a}o Vicente 225, G\'{a}vea, 22451-900 Rio de Janeiro, RJ - Brazil.}
  \and
  Alan Pereira
  \thanks{E-mail: \ alan.pereira@im.ufal.br; \ Instituto de Matem\'{a}tica, Universidade Federal de Alagoas, Rua Lorival de Melo Mota s/n, 57072970 Macei\'{o}, AL - Brazil.}
  \and
  Guilherme Reis
  \thanks{E-mail: \ ghpreis@id.uff.br; \ Departamento de Matem\'{a}tica, Universidade Federal Fluminense, Rua Professor Marcos Waldemar de Freitas Reis, s/n, 24210-201 Niterói, RJ - Brazil.}}
\maketitle

\begin{abstract}
We investigate the behavior of the empirical neighborhood distribution of marked graphs in the framework of local weak convergence. Here we extend known results by considering uniform random graphs with given degree sequences and i.i.d.\ marks on half-edges and vertices. We establish a large deviation principle for such families of empirical measures. The proof builds on Bordenave and Caputo's seminal 2015 paper, and Delgosha and Anantharam's 2019 introduction of BC entropy, relying on combinatorial lemmas that allow one to construct suitable approximations of measures supported on marked trees. Possible applications of these results are in the study of interacting diffusions on top of random graphs. 

\end{abstract}

\section{Introduction}

\subsection{Motivation.}

In probability theory, the depiction of interactions among particles or individuals often relies on the framework of random graphs. Typically, models are initially examined in the mean-field scenario, where every two particles interact in the same way. Our objective is to transcend mean-field models, extending our exploration to encompass graphs with bounded degrees.

For each \(n\in \bb N\), consider a vector \(\vec{\ell}_n=(\ell_{n,1},\dots,\ell_{n,n})\) with non-negative integer entries such that $\sum_{i=1}^n\ell_{n,i}$ is even, \(\sup_{n,i}\ell_{n,i}<\infty\), and $\frac{1}{n}\sum_{i=1}^{n} \delta_{\ell_{n,i}}$ converges weakly. We define \(\mathcal{G}(\vec{\ell}_n)\) as the set of graphs \(G\) with a vertex set \( [n]:=\{1, 2, \dots, n\}\) such that the degree of vertex \(i\) in \(G\) is \(\ell_{n,i}\) for all \(i \in [n]\). Let \(G_n\) be the uniform random graph sampled from \(\mathcal{G}(\vec{\ell}_n)\). Now, consider the marked graph \(\bar G_n\), derived from \(G_n\) by assigning independent random variables to both vertices and edges of \(G_n\).

Given a marked graph $\bar{G}_{n}$, define its empirical neighborhood distribution as \begin{equation}\label{eq:empirical_neighborhood_distribution}
U(\bar{G}_n)=\frac{1}{n}\sum_{v=1}^{n}\delta_{[\bar{G}_n(v),v]}, 
\end{equation}
where $[\bar{G}_n(v),v]$ denotes the connected component of $v$ in the graph $\bar{G}_{n}$, rooted at $v$, along with the associated fields of marks within these subgraphs.

\begin{theorem}\label{t:main_intro}
    Consider the sequence of marked graphs $\bar G_n$ as above and assume that the space of marks is finite. The sequence of empirical neighborhood measures $U(\bar{G}_n)$ satisfies a large deviation principle with speed $n$ and rate function presented in Theorem~\ref{theorem:main}.
\end{theorem}

Diverging from the mean-field scenario poses a considerable challenge, and identifying the geometric distinctions potent enough to influence processes on these graphs is not always straightforward. For instance, in the voter model, convergence of the opinion density to the Wright-Fisher diffusion holds in a broad spectrum of sparse graphs after appropriate rescaling, mirroring mean-field behavior (see~\cite{ccc}). However, the stochastic Kuramoto model exhibits notable disparities in behavior between sparse and mean-field geometries (see~\cite[Remark~6.11]{OliveiraReisStolerman2018} and~\cite{Ramanan2020}). Despite significant progress in understanding sparse geometries, particularly within the framework of local weak convergence, additional tools are imperative for attaining more nuanced insights into the dynamics of complex systems evolving within these geometries.

Our objective is to advance the arsenal of tools, aiming to replicate results akin to those in~\cite{BOPR}, but now focusing on the uniform graph with a specified degree sequence instead of the Erd\H{o}s-R\'{e}nyi random graph. Just as in~\cite{BOPR}, the BC-entropy remains a cornerstone in our methodology, yet new approximation results are indispensable. These are achieved through the employment of combinatorial lemmas, facilitating the construction of graphs with slightly varied degree sequences while minimally impacting the neighborhoods of a significant portion of vertices (cf. Lemma~\ref{le:matrix}).

\subsection{Application to interacting diffusions.}
Given a random graph $G=(V,E)$, define the Hamiltonian
\begin{equation}\label{eq:hamiltonian}
H_{G}(\vec{x}, \vec{\xi}, \vec{\omega}) = \sum_{v,u\in V} \xi_{v,u} f(x_{v}-x_{u}; \,\omega_v,\omega_u) + \sum_{v\in V} g(x_v; \omega_v),
\end{equation}
where $f$ and $g$ are two functions that represent the potential and the external field, the variables $\vec{\xi}=(\xi_{u,v} : \{u,v\} \in E) \in \R^{2|E|}$ are marks on the oriented edges that represent the strength of interaction associated to the oriented edge $(u,v)$, and $\vec{\omega}=(\omega_{v}; v \in V) \in \R^{|V|}$ are marks on vertices that represent ``media'' variables.

We write $\partial_{v}H_{G}$ for the derivative of $H_{G}$ with respect to the variable $x_{v}$.
Fix a finite time horizon $T>0$ and let $\vec{B} = (B_v: v \in V)$ be an independent family of standard Brownian motions defined on the time interval $[0,T]$.
We consider the system of interacting diffusions over $G$ given by the stochastic process $\vec{\theta}^{G}=(\theta^{G}_{v}(t); v \in V, t \in [0,T])$ that solves the following system of It\^o stochastic differential equations
\begin{equation}\label{eq:sde}
\begin{cases}
{\rm d}\theta^{G}_{v}(t) & = \partial_{v}H_G(\vec{\theta}^{G},\vec{\xi},\vec{\omega}) {\rm d}t + {\rm d}B_{v}(t), \quad 0 \leq t \leq T, \\
\theta^{G}_{v}(0) & = \theta_{v}(0), \quad \text{for all } v \in V.
\end{cases}
\end{equation}

We assume that $f$, $f'$ (derivative with respect to the $x$ variable), $f''$, $g$, $g'$, $g''$ exist, are bounded and continuous. Under this assumptions, when the marked graph $G$ is finite, the system  (\ref{eq:sde}) has a unique strong solution with continuous trajectories (see \cite[Chapter~5, Theorem~2.9]{karatzas2012brownian}). Let us denote by $\vec{\theta}_{n} = \{\theta_{n,v}(s): s \in [0,T] \text{ and } v \in [n]\}$, the solution of the Stochastic Differential Equation~\eqref{eq:sde} where $G$ is the random graph with given degree sequence introduced above. 

In the following we denote by $U(G_{n}, \vec{x}, \vec{\theta}(0), \vec{\omega}, \vec{\xi})$ for the empirical neighborhood distribution (see~\eqref{eq:empirical_neighborhood_distribution}) of the graph $G_{n}$ with oriented edges marked by the entries of the vector $\vec{\xi}$ and vertices marked by the triple $\big( \big(x_{v}(t)\big)_{t \in [0,T]}, \omega_{v}, \theta_{v}(0) \big)$, where $x_{v} \in C[0,T]$, for every $v \in G_{n}$.
\begin{theorem}
    Assume that the collections of marks $\vec{\xi}$, $\vec{\omega}$ and the initial conditions $\vec{\theta}_{n}(0)$ are obtained as i.i.d.\ copies of given arbitrary distributions. If the sequence $\big(U(G_{n}, \vec{B}, \vec{\theta}(0), \vec{\omega}, \vec{\xi})\big)_{n \in \N}$ satisfies a large deviation principle with rate function $I$, then the sequence of empirical neighborhood distributions marked with the solution of the SDE~\eqref{eq:sde}, $\big( U(G_{n}, \vec{\theta}_{n}, \vec{\theta}(0), \vec{\omega}, \vec{\xi}) \big)_{n \in \N}$, satisfies a large deviation principle with speed $n$ and rate function $I-F$, where $F$ is defined taking values of empirical neighborhood distributions as
    \begin{equation}\label{def:F}
        F(\rho) = - \int \bigg(F^{(1)}_T-F^{(1)}_0 + \frac{1}{2}F^{(2)} + \frac{1}{2}F^{(3)} \bigg) {\rm d} \rho,
    \end{equation}
and the functions $F^{(1)}_t$ for $t \in \{0,T\}$,  $F^{(2)}$, and $F^{(3)}$, are defined, for a rooted marked graph $G=(V,E,\vec{x}, \vec{\xi},\vec{\omega},o)$, as
\begin{equation}\label{def:Fi}
\begin{split}
F^{(1)}_t(G)=\sum_{w;w\sim o}& \xi_{o,u}f(x_{o}(t)-x_{u}(t);\,\omega_o,\omega_u)+g(x_o(t);\omega_o),\\
F^{(2)}(G)=\int_{0}^{T}\bigg[&\sum_{u;u\sim o}\xi_{o,u}f'(x_o(t)-x_u(t);\omega_o,\omega_u)\\
&-\sum_{u;u\sim o}\xi_{u,o}f'(x_u(t)-x_o(t);\omega_o,\omega_u)+g'(x_o(t);\omega_o)\bigg]^2{\rm d}t,\\
F^{(3)}(G)=\int_{0}^{T}\bigg[&\sum_{u;u\sim o}\xi_{o,u}f''(x_o(t)-x_u(t);\omega_o,\omega_u)\\
&+\sum_{u;u\sim o}\xi_{u,o}f''(x_u(t)-x_o(t);\omega_o,\omega_u)+g''(x_o(t);\omega_o)\bigg]{\rm d}t.
\end{split}
\end{equation}
\end{theorem}

The proof of the theorem above follows from our Theorem~\ref{t:main_intro} combined with the general discretization scheme presented in Lemma 4.2 of~\cite{BOPR} and a direct application of Theorem 7.1 of the same paper.

\subsection{Overview of the proof.}

The proof of the large deviation result for the sequence in~\eqref{eq:empirical_neighborhood_distribution} unfolds in two pivotal steps, delineated below.

In the initial step, we pivot from the complexity of $\bar{G}_n$ through a mixture argument. Instead of contending with i.i.d.\ marks, our focus shifts to marked graphs uniformly drawn from a set characterized by prescribed numbers of each mark type, denoted as $\mathcal{G}(\vec{\ell}_n)_{\vec{m}_n, \vec{u}_n}$. Here, vectors $\vec{m}_n$ and $\vec{u}_n$ count the amount vertices and edges designated for each mark type (refer to Section~\ref{sec:degreesequence} for a precise definition). After establishing a large deviation estimate for these random graphs, general results pertaining to mixtures of probability measures imply our main result.

The subsequent step consists on checking that this class of auxiliary marked graphs indeed satisfies a large deviation principle. This involves securing precise asymptotic bounds for quantities such as
\begin{equation}\label{eq:a}
|\{G\in \mathcal{G}(\vec{\ell}_n)_{\vec{m}_n,\vec{u}_n} : U(G)\in B(\rho,\delta)\}|,
\end{equation}
 where $B(\rho,\delta)$ denotes a fixed ball of radius $\delta>0$ around the distribution on marked trees $\rho$. This approach mirrors that of Delgosha and Anantharam~\cite{da}, given that the model under consideration aligns with the one studied in~\cite{da}, with the added nuance of the vector of prescribed degrees $\vec{\ell}_n$.

The primary challenge resides in utilizing the truncated version of the BC-entropy to attain the desired lower bound for \eqref{eq:a}. To accomplish this, we establish the existence of a sequence $\Gamma_n$ of marked graphs such that, when truncated at height $k$, $U(\Gamma_n)_k \to \rho_k$, concurrently maintaining $\deg_{\Gamma_n}(i)=\ell_{n,i}$. Our explicit proof hinges on the construction of a generalized configuration model for a specially crafted sequence of colored degrees. Simultaneously, we demonstrate some type of ``mass transport'' on matrices of colored degrees. Specifically, we prove that such matrices, when linked to a sequence of marked graphs, allows one to perform certain modifications that respect a series of contraints required by the model. These constraints are crucial for transitioning from a matrix of colored degrees to constructing a sequence of graphs with the desired properties, employing the framework of the generalized configuration model and the Reconstruction Lemma as detailed in~\cite[Proposition~8]{da}.

\subsection{Related works}

 Most works deal with unmarked graphs, where large deviations of several quantities have been studied, as described in the survey paper by Chatterjee~\cite{chatterjee}.
 
A lot of attention was devoted to the dense regime, where degrees grow linearly with the size of the graph and results are framed in the setting of the cut topology introduced by Lov\'{a}sz and Szegedy~\cite{ls}. Chatterjee and Varadhan~\cite{cv} establish large deviations for Erd\H{o}s-R\'{e}nyi random graphs when $p$ remains constant. Dhara and Sen~\cite{dhara_sen} consider the uniform graph with given degrees and establish a large deviation principle under mild conditions that guarantee the existence of a graphon limit, as established by Chatterjee, Diaconis and Sly~\cite{cds}. Markering~\cite{markering} extends the previous results and considers inhomogeneous Erd\H{o}s-R\'{e}nyi random graphs in the same dense setting. Finally, we mention also Dembo and Lubetzki~\cite{dl} that consider the uniform graph with given number of edges.

For the sparse regime, the framework of local weak convergence introduced by Benjamini and Schramm~\cite{bs} and Aldous and Steele~\cite{as} is widely used, since a strikingly large collection of graph functionals is continuous under this topology. In this sense, studying large deviations for empirical neighborhoods allows one to obtain similar estimates for such functionals. Large deviation estimates for graphs in this regime were first treated by Bordenave and Caputo~\cite{bc}, specifically for the Erd\H{o}s-R\'{e}nyi random graph and unifrom random graph with given degree. More recently, Backhausz, Bordenave, and Szegedy~\cite{bbs} consider the case of uniformly sampled $d$-regular random graphs as well as unimodular Galton Watson trees.

When one turns to marked graphs, the picture is more scarce. Based on~\cite{bc}, Delgosha and Anantharam~\cite{da} introduce the entropy notion that is central in our work and deduce a weak large deviation principle for the uniform marked graph with given mark counting vectors (see Section~\ref{sec:da_results}). This was then extended by Baldasso, Oliveira, Pereira, and Reis~\cite{BOPR} to cover the case of sparse marked Erd\H{o}s-R\'{e}nyi random graphs.

In the recent work~\cite{ramanan2023large} the authors show that the rate function obtained in~\cite{da,BOPR} admits a more tractable expression, involving only relative entropies, for the models of uniform marked graphs with fixed number of edges and Erd\"os-R\'enyi graphs with i.i.d.\ marks. Furthermore, they also obtain an alternative expression for the rate function of the large deviation principle of the model of uniform graph with fixed degree sequence, the same model we consider in this present work. We notice that \cite[Theorem~7.11]{ramanan2023large} relies on the proof of a large deviation principle for the model of uniform graph with fixed degree sequence and make reference to a forthcoming paper~\cite{ramanan2023large2}. The present work complements~\cite{ramanan2023large} in the sense that we prove the large deviation principle mentioned in~\cite[Theorem~7.11]{ramanan2023large}. We also mention that our work will have intersection with the announced work~\cite{ramanan2023large2}.

Many works deal with large deviation for specific quantities such as subgraph counts~\cite{augeri, bb, cd, ggs, hmw, vu}, functionals of the adjacency and Laplacian matrices~\cite{bg, hhm, markering}, and components size~\cite{aklm, akp}.

Let us delve into the landscape of research on interacting diffusions atop random graphs, one of the main potential applications of the current manuscript. In the realm of gradient dynamics on the complete graph, Dai Pra and den Hollander~\cite{dpdh} set the stage. This was later adapted by the current authors~\cite{bpr} to the non-gradient case, incorporating delays into the mix.

Transitioning to the Erd\"{o}s-R\'{e}nyi random graph, Delattre, Giacomin, and Lu\c{c}on~\cite{Delattre2016} provide insights by comparing the system with its mean-field counterpart. They establish bounds on the distances between solutions, assuming the mean degree diverges logarithmically. In the dense regime, Oliveira and Reis~\cite{or} extend these discussions with large deviation estimates for solution paths, offering a nuanced exploration. For an in-depth examination of works in the dense regime, we direct the reader to~\cite{or}.

Venturing into the sparse setting, Lacker, Ramanan, and Wu~\cite{Kavita2019} and Oliveira, Reis, and Stolerman~\cite{OliveiraReisStolerman2018} uncover the convergence of solutions to a meticulously defined model on the Galton-Watson random tree. However, their results do not explicitly delve into the realm of large deviations.

MacLaurin~\cite{2016maclaurin} contributes to the landscape by considering an interacting particle system defined through stochastic differential equations in networks that converge locally weakly to $\mathbb{Z}^d$, tailored for the sparse setting. Their large deviation principle is derived by applying transformations to an existing known result (see \cite[Section~4]{2016maclaurin}). It is worth noting that these methods appear to be tailored for networks converging locally weakly to $\mathbb{Z}^d$ and do not explicitly address other networks such as Erd\"{o}s-R\'{e}nyi or the uniform graph with given degrees in the sparse regime.

\bigskip

\noindent \textbf{Acknowledgments.}
RB has counted on the support of ``Conselho Nacional de Desenvolvimento Científico e Tecnológico - CNPq'' grants ``Projeto Universal'' (402952/2023-5) and ``Produtividade em Pesquisa'' (308018/2022-2). AP was funded by ``Alagoas Research Foundation - FAPEAL''
(E:60030.0000002397/2022) and ``Conselho Nacional de Desenvolvimento Científico e Tecnológico - CNPq'' grants ``Projeto Universal'' (402952/2023-5).
GR was partially supported while he was a post doc at Technical University of Munich (TUM).

\section{Preliminaries}\label{sec:prel}
~
\par In this section, we fix some notation that is used throughout the rest of the paper and review some results from the literature. We start by introducing the notation for marked rooted graphs, graph isomorphisms, and metrics on the set of such graphs. Furthermore, we recall the concept introduced in~\cite{as, bs} of local weak convergence and the definition of the BC-entropy introduced in~\cite{da}.

\subsection{Marked graphs}
~

A \textit{graph} $G=(V,E)$ consists of a set $V$ of vertices and a set $E$ of edges. All graphs considered here have finite or countably infinite vertex sets and are always assumed to be locally finite. A \emph{rooted graph} $(G,o)$ is a graph $G$ together with a distinguished vertex $o \in V$. We write $\sG_{*}$ the space of connected locally finite rooted graphs.

For $u,v \in V$, we write $u \sim_G v$ if these vertices are adjacent. The degree of $u$ in $G$ is denoted by $\deg_G(u)$. The induced distance by $G$ between $u$ and $v$, denoted by $\dist_{G}(u,v)$, is the size of the smallest path connecting $u$ to $v$.

\par Let $(\Theta,{\rm d}_{\Theta})$ and $(\Xi,{\rm d}_{\Xi})$ be two finite metric spaces. A \emph{marked graph} $\bar{G}=(G,\vec{\tau},\vec{\xi})$ is a graph $G=(V,E)$ endowed with fields of marks
\begin{equation}\label{eq:fields-mark}
\vec{\tau}=(\tau(v))_{v\in V} \quad \text{and} \quad \vec{\xi}=(\xi(v,w))_{\{v,w\}\in E},
\end{equation}
where $\tau(v) \in \Theta$, for all $v \in V$, and $\xi(v,w) \in \Xi$, for all $\{v,w\} \in E$. Notice that, although edges are not directed, each edge receives two marks, one for each possible orientation. We denote by $\bar{\sG}=\bar{\sG}_{(\Theta,\Xi)}$ the space of connected and locally finite marked graphs with mark spaces $\Theta$ and $\Xi$.

An \textit{isomorphism} between two rooted marked graphs $(\bar{G},o)=$ and $(\bar{G}',o)$ is a bijective map $\Psi: V \to V'$ that preserves edges, marks, and the root. For a marked graph $\bar{G}$ and a radius $r\in \N$, $(\bar{G},o)_{r}$ denotes the marked subgraph of $\bar{G}$ induced by the vertices within distance $r$ from $o \in V$.
 
Let $\bar{\sG}_{*}$ denote the space of (connected and locally finite) rooted marked graphs up to isomorphism. We often identify classes with representative elements without further mention to this. The spaces $\bar{\sG}_{*}$ and $\sG_{*}$ can be endowed with metrics ${\rm d}_{\bar{\sG}_*}$ and ${\rm d}_{\sG_{*}}$ that turn them into  Polish metric spaces (cf.~\cite{bordenave}).
 
We write $\bar{\sT}_*\subset \bar{\sG}_{*}$ for the space of (connected and locally finite) rooted marked trees up to isomorphism. Also, denote by $\bar{\sG}^{h}_{*}$ the set of rooted marked connected and locally finite graphs with depth at most $h$, meaning that all vertices are within distance $h$ from the root.
 
There exists a natural projection $\pi_g: \bar{\sG}_{*}\to\sG_{*}$ that associates to each marked graph $\bar G$ its graph component $\pi_g(\bar G)$. We will use throughtout the text that the map $\pi_{g}$ is a weak contraction, that is, a one-Lipschitz map.

Finally, let $\sG_n$ denote the set of graphs and $\bar{\sG_n}$ the set of marked graphs with vertex set $[n]$ (notice that these sets include graphs that are not connected).

\bigskip

\noindent{\textbf{Local weak convergence.}}
For a (not necessarily connected) marked graph $\bar{G}$ and a vertex $v \in V$, we associate the rooted marked graph $(\bar{G}(v),v)$ given by the marked graph induced by the connected component of $v$ rooted at $v$. We write $[\bar{G}(v),v] \in \bar{\sG}_{*}$ for the equivalence class of $(\bar{G}(v),v)$.

\par For a finite marked graph $\bar{G}$ on $n$ vertices, its \emph{empirical neighborhood distribution} is the probability measure defined in the set $\bar{\sG}_{*}$ via~\eqref{eq:empirical_neighborhood_distribution}. We will use a slight abuse of notation and also denote by $U(G)$ the empirical neighborhood distribution of an unmarked graph $G$. In this case, $U(G)$ is a probability distribution in $\sG_{*}$.

\par Denote by $\mathcal{P}\left(\bar{\sG}_{*}\right)$ the collection of probability measures on $\bar{\sG}_{*}$ endowed with the weak convergence topology. This space can be metrized by the L\'{e}vy-Prokhorov metric $\d_{LP}$, defined as follows. Denote by ${\mathcal  {B}}(\bar{\sG}_{*})$ the Borel $\sigma$-algebra of the metric space $(\bar{\sG}_*,{\rm d}_{\bar{\sG}_*})$. For each $A \in {\mathcal  {B}}(\bar{\sG}_{*})$ let 
\begin{equation*}
A^{{\varepsilon }} = \{ g \in \bar{\sG}_{*} : {\rm d}_{\bar{\sG}_*}(g,g') <\varepsilon, \text{ for some } g' \in A \}.
\end{equation*}
For $\mu, \nu \in \mathcal{P}\left(\bar{\sG}_{*}\right)$, define 
\begin{equation*}
\begin{split}
\d_{LP}(\mu ,\nu )= \inf \big\{\varepsilon >0 : \mu & (A)\leq \nu (A^{{\varepsilon }})+\varepsilon \ {\text{and}}\\ 
\nu & (A)\leq \mu (A^{{\varepsilon }})+\varepsilon \ {\text{for all}}\ A\in {\mathcal  {B}}(\bar{\sG}_{*})\big\}.
\end{split}
\end{equation*}

The following lemma is used in Section~\ref{sec:proof_t_marked_configuration_model}.

\begin{lemma}\label{l:lipschitz}
Let $S$ be a metric space. Given $f: \bar{\sG}_{*} \to E$, define $F: \mathcal{P} (\bar{\sG}_{*}) \to \mathcal{P}(S)$ as $F(\mu) = \mu \circ f^{-1}$.  If $f$ is $\alpha$-Lispchitz in the L\'{e}vy-Prokhorov metric, for some $\alpha \geq 1$, then so is the map $F$.
\end{lemma}

\begin{proof}
    Fix $\mu, \nu \in \mathcal{P}(\bar{\sG}_{*})$ and $\varepsilon>0$ such that ${\rm d}_{LP}(\mu, \nu) \leq \varepsilon$. In order to conclude the lemma, it suffices to prove that ${\rm d}_{LP}(F(\mu), F(\nu)) \leq \alpha\varepsilon$.
    
    Given a measurable set $A \subset S$, ${\rm d}_{LP}(\mu, \nu) \leq \varepsilon$ implies for the set $f^{-1}(A)$
    \begin{equation}
        \mu(f^{-1}(A)) \leq \nu ( f^{-1}(A)^{\varepsilon})+ \varepsilon \quad \text{and} \quad  \nu(f^{-1}(A)) \leq \mu ( f^{-1} (A)^{\varepsilon})+ \varepsilon.
    \end{equation}

    Notice now that, since $f$ is $\alpha$-Lipschitz,
    \begin{equation}
        f^{-1}(A)^{\varepsilon} \subseteq f^{-1}(A^{\alpha \varepsilon}).
    \end{equation}
    This then implies, since $\alpha \geq 1$,
    \begin{equation}
    F(\mu)(A) = \mu(f^{-1}(A)) \leq \nu ( f^{-1}(A)^{\varepsilon})+ \varepsilon \leq \nu ( f^{-1}(A^{\alpha \varepsilon}))+ \alpha\varepsilon = F(\nu)(A^{\alpha\varepsilon}) + \alpha\varepsilon.
    \end{equation}
    An analogous calculation yields $F(\nu)(A) \leq F(\mu)(A^{\alpha\varepsilon}) + \alpha\varepsilon$, which implies ${\rm d}_{LP}(F(\mu), F(\nu)) \leq \alpha\varepsilon$ and concludes the proof of the lemma.
\end{proof}

\begin{definition}[Local weak convergence]\label{def:localweak}
Consider a sequence $\bar{G}_n=([n],E_n,\vec{\tau}_n,\vec{\xi}_n)$ of marked graphs and let $\rho \in \mathcal{P}\left(\bar{\sG}_{*}\right)$. We say that $\bar{G}_n$ converges locally weakly to $\rho$ if $U(\bar{G}_n)$ converges to $\rho$ in the sense of weak convergence.
\end{definition}

\bigskip

\noindent{\textbf{Unimodularity.}}
Let $\bar{\sG}_{**}$ denote the set of connected marked graphs with two distinguished vertices up to isomorphisms that preserve both vertices.
 
\par A probability measure $\mu \in \mathcal{P}(\bar \sG_* )$ on the set of connected rooted marked graphs is called \textit{unimodular} if, for any measurable non–negative function $f: \bar{\sG}_{**} \rightarrow \R_+$,
\begin{equation}\label{eq:unimodularity}
\int \sum_{v\in V (G)} f([G,o,v]) \, \d\mu([G,o]) =
\int \sum_{v\in V (G)} f([G,v,o]) \, \d\mu([G,o]).
\end{equation}
The set of unimodular probability measures on $\bar{\sG}_*$ is denoted by $\mathcal{P}_u (\bar{\sG}_*)$.

For an in depth discussion of properties of unimodular measures we refer the reader to~\cite{as}. A proposition that is important for us is the following.
\begin{proposition}
The set $\mathcal{P}_u (\bar{\sG}_*)$ is closed in $\mathcal{P}(\bar \sG_* )$.
\end{proposition}

\subsection{Review of BC entropy and LDP}\label{sec:da_results}
~
\par In this section we review some of main results from~\cite{da}, that introduced a cornerstone model in order to consider more general random marked graphs. We first introduce additional notation and afterwards collect some of their results that are used throughout our text.

\bigskip

\noindent{\textbf{Count vectors and degrees.}} In this entire section, we assume that the metric spaces $\Theta$ and $\Xi$ are endowed with an order $\leq$. Notice that this is possible due to the fact that the spaces are finite.

Let $G$ be a finite marked graph. We define the \emph{edge-mark count vector} of $G$ by $\vec{m}_G := (m_G (x,x') \,:\, x,x'\in \Xi)$, where $m_G (x,x')$ is the number of \emph{oriented edges} with associated marks $x$ and $x'$, i.e., number of oriented edges $(v,w)$ such that $\xi_{G}(v,w) = x$ and $\xi_{G}(w,v)=x'$  or $\xi_{G}(v,w) = x'$ and $\xi_{G}(w,v)=x$. Analogously, the \emph{vertex-mark count vector} of $G$ is $\vec{u}_G := (u_G (\theta) : \theta \in \Theta)$, where $u_G (\theta)$ is the number of vertices $v \in V (G)$ with $\tau_G (v) = \theta$.

Let us now introduce some more notation that will allow us to explore the orientation used for edge marks. For a pair $(x,y) \in \Xi \times \Xi$, we write $(x,y)_{\leq} = (x,y)$ if $x \leq y$ and $(x,y)_{\leq} = (y,x)$, otherwise. We write $\chi^{2}_{\leq}$ for the distribution of $X_{\leq}$, if $X \in \Xi^{2}$ is sampled according to $\chi \otimes \chi$. We also denote by
\begin{equation}\label{eq:xi_2_leq}
\Xi^{2}_{\leq} = \{ (x, x') \in \Xi^{2} : x \leq x'\}.
\end{equation}

\par For a vertex $o \in V (G)$, $\deg^{x,x'}_G(o)$ denotes the number of vertices $v$ connected to $o$ in $G$ such that $\xi_G (v,o) = x$ and $\xi_G (o,v) = x'$. Notice that $\deg_G (o) = \sum_{x,x'\in \Xi} \deg^{x,x'}_G(o)$.

For $\mu \in \mathcal{P}(\bar \sG_*)$, and $\theta \in \Theta$, we define
\begin{equation*}
\Pi_\theta (\mu)= \mu (\tau_G(o)=\theta) \mbox{ and } \vec{\Pi}(\mu)=(\Pi_\theta(\mu):\theta \in \Theta).
\end{equation*}
Also, for $x,x' \in \Xi$, let
\begin{equation}\label{eq:mean-degrees}
\begin{split}
\deg^{x,x'} (\mu) & = \E_{\mu}(\deg^{x,x'}_G(o)),\\
\deg(\mu )& = \sum_{x,x'}\deg^{x,x'} (\mu),\\
\vec{\deg}(\mu) & = (\deg^{x,x'}(\mu)).
\end{split}
\end{equation}

\bigskip

\noindent{\textbf{The DA model.}} Let us now introduce the model considered in~\cite{da}. Fix a sequence of edge-mark count vectors and a sequence of vertex-mark count vectors
\begin{equation}\label{eq:def_u_m}
\begin{split}
\vec{u}_n&=(u_n(\theta) : \theta \in \Theta),\\
\vec{m}_n&=(m_n(x,x') : x,x' \in \Xi) \mbox{ with } m_n(x,x')=m_n(x',x).
\end{split}
\end{equation}

\begin{remark}
Notice that in order to determine an edge-mark count vector $m_n$ it is only necessary to characterize $m_{n,{\leq}}$, by setting, for $x \leq y$, $m_{n,\leq}(x,y) = m_n(x,y) = m_n(y,x)$, by the symmetry assumption in the equation above.
\end{remark}

For $\vec{m}_n$, we define
\begin{equation}\label{def:norm1}
\big\|\vec m_n \big\|_{1} = \frac{1}{2}\sum_{x \neq x'} m_n(x,x')+\sum_{x} m_n(x,x) = \sum_{x \leq x'}m_n(x,x'),
\end{equation}
and, for $\vec{u}_n$, we have $\pnorm{1}{\vec u} = \sum_{\theta \in \Theta} u_n(\theta)$.

For each $n\in \N$, define
\begin{equation*}
\sG_{\vec m_n,\vec u_n}=\{G \in \bar{\sG}_{n} : \vec{u}_G=\vec u_n \text{ and } \vec{m}_G=\vec m_n\},
\end{equation*}
the collection of graphs with vertex- and edge-mark vectors given by $\vec u_n$ and $\vec{m}_n$, respectively.

Let $G_n$ be uniformly sampled from $\sG_{\vec m_n,\vec u_n}$, and recall that $U(G_{n})$ denotes its empirical neighborhood distribution, see~\eqref{eq:empirical_neighborhood_distribution}. In~\cite{da}, the authors introduce a notion of entropy for such random graphs, which they call BC entropy, and provide a weak large deviation principle for the sequence of random probability measures $\big( U(G_{n}) \big)_{n \in \N}$.

The asymptotic behavior of the quantities of interest depend on the limits of the vectors $\vec{u}_n$ and $\vec{m}_n$. With this in mind, we introduce the average-degree vector as
\begin{equation*}
\vec{d}=\Big( d_{x,x'}\in \R_+ : x,x' \in \Xi, d_{x,x'}=d_{x',x}, \sum_{x,x'}d_{x,x'}>0 \Big).
\end{equation*}
The vector $\vec{d}$ plays the role of the limit of $\vec{m}_n/n$. We also introduce a distribution $Q=(q_\theta)_{\theta\in \Theta}$ over $\Theta$. Likewise, $Q$ will be the limit of $\vec{u}_n/n$.

\begin{definition}\label{def:adapted}
Given an average-degree vector $\vec d$ and a probability distribution $Q = (q_\theta)_{\theta \in \Theta}$, we say that a sequence $(\vec m_n ,\vec u_n)$ of edge- and vertex-mark count vectors is adapted to $( \vec d,Q)$, if the following conditions hold:
\begin{enumerate}
\item $\pnorm{1}{\vec m_n}\leq \binom{n}{2}$ and $\pnorm{1}{\vec u}=n$. This guarantees that $\sG_{\vec m_n,\vec u_n}$ is not empty.
\item $m_n(x,x)/n\to d_{x,x}/2$, for all $x \in \Xi$.
\item $m_n(x,x')/n\to d_{x,x'}=d_{x',x}$, for all $x \neq x' \in \Xi$.
\item $u_n(\theta)/n\to q_\theta$, for all $\theta \in \Theta$.
\item If $d_{x,x'}=0$, then $m_n(x,x')=0$, for all $n$.
\item $q_\theta=0$ implies $u_n(\theta)=0$, for all $n$.
\end{enumerate}
\end{definition}

\bigskip

\noindent{\textbf{The large deviation principle.}} Let $\mu \in \mathcal{P}(\bar \sG_*)$ and denote by $B(\mu,\varepsilon)$ the ball of radius $\varepsilon$ in the L\'{e}vy-Prokhorov metric in $\mathcal{P}(\bar \sG_*)$. The large deviation result concerns the exponential decay behavior of the sequence of probabilities
\begin{equation*}
\P(U(G_n)\in B(\mu,\varepsilon)).
\end{equation*}

Consider the set
\begin{equation*}
\sG_{\vec m_n,\vec u_n}(\mu,\varepsilon) = \left\{ G\in \sG_{\vec m_n,\vec u_n} : \d_{LP}(U(G),\mu)<\varepsilon \right\}.
\end{equation*}
Since $G_n$ is uniformly distributed over $\sG_{\vec{m}_n,\vec{u}_n}$, it follows by definition that
\begin{equation*}
\log \P(U(G_n)\in B(\mu,\varepsilon))=\log \big|\sG_{\vec m_n,\vec u_n}(\mu,\varepsilon)\big|-\log \big|\sG_{\vec m_n,\vec u_n}\big|.
\end{equation*}

If $(\vec m_n ,\vec u_n)$ is adapted to $(\vec d, Q)$, it is proved in \cite[Equation~10]{da} that 
\begin{equation}\label{eq:sizeGmu}
\log \big|\sG_{\vec m_n,\vec u_n}\big|=\big\|{\vec m_n}\big\|_{1}\log n+nH(Q)+n\sum_{x,x'}s(d_{x,x'})+o(n),
\end{equation}
where $H(Q) = -\sum_{\theta \in \Theta} q_{\theta} \log q_{\theta}$ denotes the entropy of the probability measure $Q$ and $s(d)$ is given by
\begin{equation}
s(d) = \begin{cases}
\frac{d}{2}-\frac{d}{2} \log d, & \quad \text{ if } d>0, \\
0, & \quad \text{ if } d=0.
\end{cases}
\end{equation}

Delgosha and Anantharam~\cite{da} introduced the notion of $\varepsilon$-upper BC entropy
\begin{equation}\label{eq:def_entropy}
\overline \Sigma_{\vec d,Q}(\mu,\varepsilon)\vert_{\vec m_n,\vec u_n} = \limsup_{n\to\infty}\frac{\log |\sG_{\vec m_n,\vec u_n}(\mu,\varepsilon)|-\pnorm{1}{\vec m_n}\log n}{n}.
\end{equation}
Taking the limit $\varepsilon\to 0$, the upper BC entropy is defined as
\begin{equation*}
\overline \Sigma_{\vec d,Q}(\mu)\vert_{\vec m_n,\vec u_n}=\lim_{\varepsilon \downarrow 0}\overline \Sigma_{\vec d,Q}(\mu,\varepsilon)\vert_{\vec m_n,\vec u_n}.
\end{equation*}
Similar definitions hold when we exchange the superior limits for inferior limits and, in this case, the lower BC entropy is denoted by $\underline{\Sigma}$.

Recall that $s(d)=d/2-(d/2)\log d$, if $d>0$, and $s(d)=0$ if $d=0$. For an average-degree vector $\vec{d}$, we denote by $s(\vec{d})$ the quantity
\begin{equation}\label{def:s}
s(\vec{d})=\sum_{x,x'}s(d_{x,x'}).
\end{equation}

The next result establishes the central piece of the large deviation principle. It proves that the upper and lower BC entropies coincide and that they do not depend on the adapted sequence $(\vec{m}_n, \vec{u}_n)$.
\begin{theorem}[\cite{da}, Theorem~2]\label{t:equality_entropy}
Let $\mu\in \mathcal{P}(\bar\sG_*)$ with $0<\deg(\mu)<\infty$. The following statements hold.
\begin{enumerate}
\item The values of $\overline \Sigma_{\vec d,Q}(\mu)\vert_{\vec m_n,\vec u_n}$ and $\underline \Sigma_{\vec d,Q}(\mu)\vert_{\vec m_n,\vec u_n}$ do not depend on the adapted sequence $(\vec{m}_n,\vec{u}_n)$. For this reason, we simplify the notation by writing $\overline \Sigma_{\vec d,Q}(\mu)$ and $\underline \Sigma_{\vec d,Q}(\mu)$. 
\item The equality $\overline \Sigma_{\vec d,Q}(\mu)=\underline \Sigma_{\vec d,Q}(\mu)$ holds. This value is simply written as $ \Sigma_{\vec d,Q}(\mu)$ and it belongs to $[-\infty,s(\vec d)+H(Q)]$.
\end{enumerate}
\end{theorem}

The quantity $\Sigma_{\vec d,Q}(\mu)$ is called the BC entropy of $\mu$ associated with the pair $(\vec d,Q)$. From \cite[Theorem~1]{da}, one concludes that, unless $\vec d = \vec \deg(\mu)$, $Q =\vec{\Pi}(\mu)$, and $\mu$ is a unimodular measure on $\bar\sT_*$, $\Sigma_{\vec d,Q} (\mu) = -\infty$.

Combining the theorem above, Equation~\eqref{eq:sizeGmu} and~\cite[Theorem 4.1.11]{dembo_zeitouni}, one readily concludes a weak large deviation principle for the collection $\big(U(G_{n})\big)_{n \in \N}$, with rate function given by
\begin{equation}\label{eq:rate_function_uniform_marked_graph_0}
I_{\vec{d},Q}(\mu)= H(Q)+s(\vec{d})-\Sigma_{\vec d,Q} (\mu).
\end{equation}
This result was further strengthened in~\cite[Theorem~2.8]{BOPR} to a complete large deviation principle.

\section{Main result and proof}\label{sec:proof_t_marked_configuration_model}

In this section, we give a precise description of the rate function for the large deviation satisfied by the sequence of empirical neighborhood distributions defined in~\eqref{eq:empirical_neighborhood_distribution}.

Throughout the paper we assume that the degree sequence $\vec{\ell}_n$ satisfies
\begin{equation}\label{eq:uniformly_bounded}
    \sup_{n\geq 1} \sup_{1\leq i \leq n} \ell_{n,i} < \infty\,,
\end{equation} 
and that there exists a probability measure $P$ on $\bb{Z}_+$ such that the following convergence holds in the weak sense
\begin{equation}\label{eq:b}
P_{n} = \frac{1}{n}\sum_{i=1}^{n} \delta_{\ell_{n,i}} \to P. 
\end{equation}

Let $G_n$ be a random marked graph uniformly sampled from $\mathcal{G}(\vec{\ell}_n)$ with i.i.d.\ vertex marks sampled according to $\vartheta \in \mathcal{P}(\Theta)$ and i.i.d.\ edge marks with distribution $\chi \in \mathcal{P}(\Xi)$, that is, for $\bar{G} = (G,\vec{\theta}, \vec{\xi}) \in \bar{\mathcal{G}}(\vec{\ell}_n)_{\vec{m}_n, \vec{u}_n}$, let
\begin{align}\label{def:model}
    \P(G_n = \bar{G}) =  \frac{1}{|\mathcal{G}(\vec{\ell}_n)|}\cdot \P( \vec{O} = \vec{\theta}) \cdot  \P ( \vec{X} = \vec{\xi}),
\end{align}
where the coordinates of $\Vec{X}$ are i.i.d.\ with distribution $\chi$ and the coordinates of $\vec{O}$ are i.i.d.\ with distribution $\vartheta$.

Define now 
\begin{equation} \label{eq:J_1}
J_1(P) = \lim_{n\to\infty}\frac{1}{n}\big(\log |\mathcal{G}(\vec{\ell}_n)| - \big\|\vec m_n \big\|_{1}\log n\big),
\end{equation}
the quantity that controls the asymptotic growth of $|\mathcal{G}(\vec{\ell}_n)|$.

\begin{remark}
We point out that~\cite[Equation~45]{da} establishes the equality
\begin{equation*}
J_1(P) = \Sigma({\rm UGWT}_1(P)),    
\end{equation*}
where ${\rm UGWT}_1(P)$ denotes the distribution of the unimodular Galton-Watson tree with offspring distribution $P$ and unitary spaces of marks, defined in~\cite{da}.
\end{remark}

Finally, set
\begin{equation}\label{eq:ratefunction_l:6.1}
I_{P,\vec d,Q}(\rho) = 
\begin{cases}
J_1(P)+I_{\vec{d},Q}(\rho), & \text{ if } \pi_g(\rho_1) = P, \\
+\infty, & \text {otherwise},
\end{cases}
\end{equation}
with $I_{\vec{d},Q}(\rho)$ given in~\eqref{eq:rate_function_uniform_marked_graph_0}.

We are now in position to precisely state the content of~\eqref{eq:empirical_neighborhood_distribution}. Recall the definition of $\chi^{2}_{\leq}$ in the paragraph above Equation~\eqref{eq:xi_2_leq}.
\begin{theorem}\label{theorem:main}
    Assume that the degree sequence $(\vec{\ell}_{n})_{n \in \N}$ satisfies~\eqref{eq:uniformly_bounded} and~\eqref{eq:b}. Let $G_n$ be the random marked graph with distribution given by~\eqref{def:model} and set $d  = \sum_{\ell=0}^{\infty} \ell P(\ell)$. 
    Then the sequence $(U(G_n))_{n\geq 1}$ satisfies a large deviation principle with with rate function
    \begin{equation} \label{eq:ratefunction_main}
\lambda_{P}(\mu) =I_{P, \vec\deg(\mu), \vec{\Pi}(\mu)}(\mu) + \frac{d}{2} H\Big( \frac{1}{d}\vec\deg(\mu)_{\leq} \Big| \chi_\leq^{2} \Big) + H( \vec{\Pi}(\mu) |\vartheta ),
\end{equation}
if $\deg(\mu)=d$ and $\lambda_{P}(\mu) = +\infty$ otherwise.
\end{theorem}

\begin{remark}
In the theorem above we denote by $\vec{\deg}(\mu)_{\leq}$ for the vector given by $\deg^{x,x'}(\mu)_{\leq} = \deg^{x,x'}(\mu)+\deg^{x',x}(\mu)$, for $x < x'$, and $\deg^{x,x}(\mu)_{\leq} = \deg^{x,x}(\mu)$. Notice that, in the case when $\deg(\mu) = d$, $\frac{1}{d}\vec{\deg}(\mu)_{\leq}$ defines a probability measure on $\Xi^{2}_{\leq}$.
\end{remark}

The proof of the result above is split into three main parts. Section~\ref{sec:degreesequence} contains a large deviation result for an intermediate model that is used in combination with mixture tools introduced in Section~\ref{sec:mixture}, where the proof of the theorem is completed. Section~\ref{subsec:proof_lemma_modification} contains a proof of a technical lemma used in the proof.

\bigskip

\subsection{Uniform marked graphs with given degree sequence.}\label{sec:degreesequence}
~

The first step towards Theorem~\ref{theorem:main} is to consider the intermediate model from which we construct our mixture representation.

We consider the random graph with given degree sequences and vectors of edge- and vertex-marks. In order to introduce the model precisely, recall that $\mathcal{G}(\vec{\ell}_n)$ denotes the set of graphs with $n$ vertices and such that $\deg_{G}(i) = \ell_{n,i}$.
Notice that the number of edges in each graph in $\mathcal{G}(\vec{\ell}_n)$ is exactly $\frac{1}{2} \sum_{i=1}^{n}\ell_{n,i}$. Let furthermore $(\vec{u}_n, \vec{m}_n)$ be sequences of vertex- and edge-marks vectors in a way that $\big\|\vec m_n \big\|_{1} = \frac{1}{2} \sum_{i=1}^{n}\ell_{n,i}$. Finally, define the set of marked graphs
\begin{equation}
\bar{\mathcal{G}}(\vec{\ell}_n)_{ \vec{m}_n,\vec{u}_n}=\{\bar{G}=(G,\vec{\tau},\vec{\xi}) \in \bar{\mathcal{G}}_{n} : G \in \mathcal{G}(\vec{\ell}_n), \vec{u}_{\bar{G}}=\vec{u}_n, \vec{m}_{\bar{G}}=\vec{m}_n\}.
\end{equation}
We consider the random graph $\bar{G}_{n}$ uniformly sampled from $\bar{\mathcal{G}}(\vec{\ell}_n)_{ \vec{m}_n,\vec{u}_n}$. Recall the definition of $P_{n} = \frac{1}{n}\sum_{i=1}^{n} \delta_{\ell_{n,i}}$ from ~\eqref{eq:b}. 

\begin{theorem} \label{ldp:uniform_given_neihborhood}
Let $\bar{G}_n$ be uniformly sampled from $\bar{\mathcal{G}}(\vec{\ell}_n)_{ \vec{m}_n,\vec{u}_n}$ and assume that $P_{n} \rightarrow P$, where $P \in \mathcal{P}(\Z_{+})$ is a probability measure with finite support. If $(\vec{m}_n, \vec{u}_n)$ is adapted to $(\vec{d}, Q)$, then the sequence $\big( U(\bar{G}_n) \big)_{n \in \N}$ satisfies a large deviation principle with rate function $I_{P,\vec d,Q}$ given by~\eqref{eq:ratefunction_l:6.1}.
\end{theorem}

Before presenting the proof of the theorem above, we state a preliminary result. This next lemma allows for the modification of the finite neighborhoods of the graphs in a way that the empirical measures of the degrees still converge to $P$ but the analogous limit for a higher depth $k>h$ equals another target distribution. Recall $U( \bar{G} )_k$ denotes the depth $k$ empirical neighborhood distribution of a (marked) graph $\bar{G}$.

\begin{lemma}\label{lemma:modification}
Given an integer $k > 1$ and $\rho \in \mathcal{P} (\bar{\mathcal{T}}_*)$ supported on trees with degrees uniformly bounded by $L>0$ and $\rho_{1} \circ \pi_{g}^{-1} = P$, there exists a sequence of marked graphs $\big( \widetilde{\Gamma}_n \big)_{n \in \N}$ such that $U(\widetilde{\Gamma}_n)_k \rightarrow \rho_k$ and $U(\pi_{g}(\widetilde{\Gamma}_n))_1 = P_{n}$, for all $n \in \N$.
\end{lemma}

The proof of the lemma above is postponed to Section~\ref{subsec:proof_lemma_modification}. We now prove Theorem~\ref{ldp:uniform_given_neihborhood}.

\begin{proof}[Proof of Theorem~\ref{ldp:uniform_given_neihborhood}]
In order to establish the theorem, it suffices to verify that the sequence $\big( U(\bar{G}_{n}) \big)_{n \in \N}$ is exponentially tight and that, for every $\rho \in \mathcal{P}(\bar{\mathcal{G}}_{*})$,
\begin{equation}\label{eq:prob_bola_2}
\lim_{\delta \to 0} \lim_{n \to \infty} \frac{1}{n} \log \mathbb{P}(U(\bar{G}_n)\in B(\rho,\delta)) = -I_{P,\vec d,Q}(\rho),
\end{equation}
where $B(\rho, \delta)$ denotes the ball of radius $\delta$ around $\rho$ in the L\'{e}vy-Prokhorov metric and $I_{P,\vec d,Q}$ is given by~\eqref{eq:ratefunction_l:6.1}.

In order to establish exponential tightness for the sequence, let $L$ be the maximum degree of the sequence $(\vec{\ell}_n)_{n \in \N}$, which is finite by assumption. Observe that the maximum degree of the graphs in the sequence $\big( \bar{G}_{n} \big)_{n \in \N}$ is also uniformly bounded by $L$. This immediately implies that $\big( U(\bar{G}_{n}) \big)_{n \in \N}$ is contained in the set of probability measures supported on the set $\mathcal{K}_{L}$, the set of marked graphs with maximum degree $L$. Exponential tightness follows from the fact that $\mathcal{K}_{L}$ is a compact set, which in turn follows from \cite[Lemma~ B.1]{BOPR} and \cite[Lemma~2.1]{bc}.

It remains to verify~\eqref{eq:prob_bola_2}. This will be done two steps, by first checking the upper bound and then the lower bound.   Since $\bar{G}_{n}$ is sampled uniformly in $\bar{\mathcal{G}}(\vec{\ell}_n)_{ \vec{m}_n,\vec{u}_n}$, we have
\begin{equation}\label{eq:expressionP2}
\mathbb{P}(U(\bar{G}_n)\in B(\rho,\delta)) = \frac{|\{\bar{G}\in \bar{\mathcal{G}}_{\vec{m}_n,\vec{u}_n} : U(\pi_{g}(\bar{G}))_1 = P_{n}, U(\bar{G})\in B(\rho,\delta)\}|}{|\bar{\mathcal{G}}(\vec{\ell}_n)_{ \vec{m}_n,\vec{u}_n}|}.
\end{equation}

We start with the upper bound, by first analysing the denominator in the fraction above. Recall from~\cite[Equation 9]{da} that
\begin{equation}\label{eqDenominator}
|\bar{\mathcal{G}}(\vec{\ell}_n)_{ \vec{m}_n,\vec{u}_n}|=|\mathcal{G}(\vec{\ell}_n)||T(\vec{u}_n/n)||T(\vec{m}_n/m_{n})|\times 2^{\sum_{x<x'}m_n(x,x')},
\end{equation}
where $T(\nu)$ denotes the type class of the measure $\nu$ \begin{equation*}
T_n(\nu) = \{ \vec{\theta} \in \Theta^n: L_n(\vec{\theta}) = \nu \}.
\end{equation*}
and analogously for the space $\Xi^{m_{n}}$. Combining the above with the definition of $J_1(P)$ in~\eqref{eq:J_1} yields
\begin{equation}\label{eqDenominator2}
\log |\bar{\mathcal{G}}(\vec{\ell}_n)_{ \vec{m}_n,\vec{u}_n}|=nJ_1(P)+m_n\log n+nH(Q)+ns(\vec{d})+o(n),
\end{equation}
see~\cite[Appendix~G]{da}.

We now work on the numerator of~\eqref{eq:expressionP2}. Consider two cases, depending on whether $\rho_{1} \circ \pi_{g}^{-1} = P$ or not.
 
Assume first that $\rho_{1} \circ \pi_{g}^{-1} \neq P$ and set $\gamma = \d_{LP}(P, \rho_{1} \circ \pi_{g}^{-1}) >0$. Let us check that the numerator in~\eqref{eq:expressionP2} is zero if $\delta$ is small enough, which already implies the equality~\eqref{eq:prob_bola_2} with $I(\rho) = \infty$. Let $\delta < \gamma/2$ and notice that, if $U(\bar{G}) \in B(\rho, \delta)$, one can apply Lemma~\ref{l:lipschitz} with the maps $\rho \mapsto \rho_{1}$ and $\rho \mapsto \rho \circ \pi_{g}^{-1}$ to obtain
\begin{equation}\label{eq:numerator1}
    \gamma/2 \geq \d_{LP}(U(\bar{G}), \rho) \geq  \d_{LP}(U(\bar{G})_1, \rho_1) \geq \d_{LP} ( P_{n} , \rho_1 \circ \pi_{g}^{-1} ).
\end{equation}
On the other hand, triangular inequality implies
\begin{equation*}
    \d_{LP}( P_{n} , \rho_1 \circ \pi_{g}^{-1} ) \geq \d_{TV}(P, \rho_1 \circ \pi_{g}^{-1}) - \d_{LP}(P, P_{n}) = \gamma - \d_{LP}(P_{n}, P).
\end{equation*}
Since $P_{n} \rightarrow P$, for $n$ sufficiently large we have $\d_{LP}(P_{n}, P) < \tfrac{\gamma}{2}$, and thus
\begin{equation*}
    \d_{LP}(P_{n}, \rho_1 \circ \pi_{g}^{-1}) > \frac{\gamma}{2},
\end{equation*}
contradicting~\eqref{eq:numerator1}. Therefore, for $n$ sufficiently large, the set in  numerator of $\eqref{eq:expressionP2}$ is empty, and the probability in~\eqref{eq:prob_bola_2} equals zero.

Let us now work on the case when $\rho_{1} \circ \pi_{g}^{-1} = P$. Fixed $k > 1$, we use Lemma~\ref{lemma:modification} to construct a sequence of marked graphs $\big( \widetilde{\Gamma}_n \big)_{n \in \N}$ such that $U(\widetilde{\Gamma}_n)_k \Rightarrow \rho_k$ and $U(\pi_{g}(\widetilde{\Gamma}_n))_1 = P_{n}$, for all $n \in \N$. We define 
\begin{equation}
    N_k(\widetilde{\Gamma}_n) = |\{ \bar{G} \in \bar{\mathcal{G}}_{\vec{m}_n,\vec{u}_n} : U(\bar{G})_k = U(\widetilde{\Gamma}_n)_k\}|\,.
\end{equation}
Furthermore, the limit below exists and we use it as the definition of  $J_k(\rho_k)$:
\begin{equation}\label{eq:ldp_proof_eq2}
J_k(\rho_k) = \lim_{n\to \infty}\frac{1}{n}\big(\log N_k(\widetilde{\Gamma}_n)-m_n\log n\big),
\end{equation}
see \cite[Equation~45]{da} for an alternative definition of $J_k(\rho_k)$.

The first step is to recover~\cite[Equation~69]{da}, which states that, for each $k \geq 1$, \begin{equation}\label{eq:69da}
\lim_{\delta\to 0}\limsup_{n\to\infty}\frac{1}{n}\big(\log |\{\bar{G}\in \bar{\mathcal{G}}_{\vec{m}_n, \vec{u}_n} : \d_{LP}(U(\bar{G})_k,\rho_k)\leq \delta \}|-m_n\log n\big)\leq J_k(\rho_k).
\end{equation}
By taking $k$ large enough, we deduce via Lemma~\ref{l:lipschitz}
\begin{equation}
\begin{split}
\lim_{\delta \to 0} & \limsup_{n \to \infty}\frac{1}{n}\big(\log |\{\bar{G} \in \bar{\mathcal{G}}_{\vec{m}_n,\vec{u}_n} : U(\pi_{g} (\bar{G}))_1 = P_{n}, U(\bar{G})\in B(\rho,\delta)\}|-m_n\log n\big) \\ & \qquad \leq \lim_{\delta \to 0} \limsup_{n \to \infty}\frac{1}{n}\big(\log |\{\bar{G}\in \bar{\mathcal{G}}_{\vec{m}_n,\vec{u}_n} : {\rm d}_{LP}(U(\bar{G})_k,\rho_k) \leq \delta \}| - m_n\log n\big) \\
& \qquad \leq J_k(\rho_k).
\end{split}
\end{equation}
Since $\lim_{k} J_{k}(\rho_{k}) = \Sigma(\rho)$ (see \cite[Theorem 3]{da}), we now let $k$ grow and combine the equation above with~\eqref{eqDenominator2} to obtain
\begin{equation}
\lim_{\delta \to 0} \limsup_{n \to \infty} \frac{1}{n} \log \mathbb{P}(U(\bar{G}_n)\in B(\rho,\delta)) \leq \Sigma(\rho)-J_{1}(P)-H(Q) - s(\vec{d}).
\end{equation}
This concludes the verification of the upper bound in~\eqref{eq:prob_bola_2}.

It now remains to prove a matching lower bound for~\eqref{eq:prob_bola_2}. Notice that in this case we can assume $I_{P, \vec{d}, Q} (\rho)< \infty$, which implies that $\rho_{1}$ is supported on stars with degree bounded by $L$.

We now bound
\begin{align*}
|\{\bar{G}\in \bar{\mathcal{G}}_{\vec{m}_n,\vec{u}_n} & : U(\pi_{g} (\bar{G}))_1 = P_{n}, U(\bar{G}) \in B(\rho,\delta)\}| \\
& \geq |\{\bar{G}\in \bar{\mathcal{G}}_{\vec{m}_n,\vec{u}_n} : U(\pi_{g} (\bar{G}))_k=U( \pi_{g}(\widetilde{\Gamma}_n))_k, U(\bar{G})\in B(\rho,\delta)\}| \\
& \geq |\{\bar{G}\in \bar{\mathcal{G}}_{\vec{m}_n,\vec{u}_n} : U(\bar{G})_k=U( \widetilde{\Gamma}_n)_k, U(\bar{G})\in B(\rho,\delta)\}|.
\end{align*}

Let $\widetilde{G}_{k,n}$ be a marked graph uniformly sampled from the set $\{ \bar{G} \in \bar{\mathcal{G}}_{\vec{m}_n,\vec{u}_n} : U(\bar{G})_k = U(\widetilde{\Gamma}_n)_k\}$. Dividing and multiplying by $N_k(\widetilde\Gamma_n)$ we obtain 
\begin{equation}\label{eq:ldp_proof_eq1}
|\{\bar{G}\in \bar{\mathcal{G}}_{\vec{m}_n, \vec{u}_n} : U(\bar{G})_k = U(\widetilde{\Gamma}_n)_k, U(\bar{G})\in B(\rho,\delta)\}|=N_k(\widetilde{\Gamma}_n)\mathbb{P}\big(U(\widetilde{G}_{k,n})\in B(\rho,\delta)\big)\,.
\end{equation}
The proof will then be completed by verifying that, for every $\delta>0$ fixed,
\begin{equation}\label{eq:tilted_prob_limit}
\lim_{k \to \infty} \lim_{n \to \infty} \mathbb{P}\big(U(\widetilde{G}_{k,n})\in B(\rho,\delta)\big) = 1,
\end{equation}
since this immediately implies, when combined with~\eqref{eqDenominator2},~\eqref{eq:ldp_proof_eq1},~\eqref{eq:ldp_proof_eq2}, and taking the limit as $k$ grows,
\begin{equation}
\lim_{\delta \to 0} \lim_{n \to \infty} \frac{1}{n} \log \mathbb{P}(U(\bar{G}_n)\in B(\rho,\delta)) \geq \Sigma(\rho)-J_{1}(P)-H(Q) - s(\vec{d}) = -I_{P, \vec{d}, Q} (\rho).
\end{equation}

In order to check~\eqref{eq:tilted_prob_limit}, choose $k \geq 3\delta^{-1}$ and recall that  $U(\widetilde{G}_n)_k=U(\widetilde{\Gamma}_n)_k$. Since $U(\widetilde{\Gamma}_n)_k \rightarrow \rho_k$, $\d_{LP}(U(\widetilde{G}_n)_k,\rho_k) = \d_{LP}(U(\widetilde{\Gamma}_n)_k,\rho_k) \leq \delta/3$, for all $n$ sufficiently large. As consequence,
\begin{equation}
\begin{split}
{\rm d}_{LP}(U(\widetilde{G}_n),\rho)&\leq \d_{LP}(U(\widetilde{G}_n),U(\widetilde{G}_n)_k)+{\rm d}_{LP}(U(\widetilde{G}_n)_k,\rho_k)+{\rm d}_{LP}(\rho_k,\rho)\\
&\leq \d_{LP}(U(\widetilde{G}_n),U(\widetilde{G}_n)_k)+{\rm d}_{LP}(U(\widetilde{\Gamma}_n)_k,\rho_k)+{\rm d}_{LP}(\rho_k,\rho)\\
&\leq \frac{1}{1+k}+\frac{\delta}{3}+\frac{1}{1+k}\\
&<\delta,
\end{split}
\end{equation}
which implies~\eqref{eq:tilted_prob_limit} and concludes the proof.
\end{proof}

\bigskip

\subsection{Proof of Theorem~\ref{theorem:main}.}\label{sec:mixture}
~

In order to conclude the proof of our main theorem, we now construct a mixture structure. In fact, we prove that the empirical distribution of the random graph with given degrees and i.i.d.\ fields of marks can be written as a mixture of graphs sampled uniformly from $\bar{\mathcal{G}}(\vec{\ell}_n)_{\vec{m}_n, \vec{u}_n}$, when the vectors $\vec{m}_n$ and $\vec{u}_n$ are randomly chosen.

Let $L_{m_{n}}(\chi^{2}_{\leq}) = \frac{1}{m_{n}}\sum_{i=1}^{m_{n}}\delta_{X^{(i)}_{\leq}}$ and  $L_{n}(\vartheta) = \frac{1}{n}\sum_{i=1}^{n} \delta_{Y_{i}}$ denote the empirical means of i.i.d.\ variables with respective distributions $\chi^{2}_{\leq}$ and $\vartheta$. 

\begin{lemma}\label{le:themix}
For any $G \in \bar{\mathcal{G}}(\vec{\ell}_n)_{\vec{m}_n,\vec{u}_n}$, it holds
\begin{equation}\label{eq:uufmixuif}
\P(G_n=G) = 
\P \big( \pi_g(G_n)= \pi_g(G) \big) \P \bigg( L_{m_n}(\chi^{2}_{\leq})=\frac{\vec m_{n,\leq}}{m_n} \bigg)  \P \bigg( L_n(\vartheta)=\frac{\vec u_n}{n} \bigg).
\end{equation}
\end{lemma}

The proof of the lemma above follows exactly the same steps as in \cite[Lemma 5.1]{BOPR} and we choose to omit it here.

Observe that the right-hand side of Equation~\eqref{eq:uufmixuif} is zero unless $\vec{m}_n=\vec{m}_G$ and $\vec{u}_n=\vec{u}_G$. This implies that the distribution of $G_n$ is a mixture of the uniform graph in $\bar{\mathcal{G}}(\vec{\ell}_n)_{\vec{m}_n,\vec{u}_n}$ when the empirical measure of the edge- and vertex-marks are given by $ L_{m_n}(\chi^{2}_{\leq})$ and $L_n(\vartheta)$. As a consequence, $U(G_n)$ is a mixture of $U(H_n)$ where $H_n$ is a uniform graph in $\bar{\mathcal{G}}(\vec{\ell}_n)_{\vec{m}_n,\vec{u}_n}$. Therefore a large deviation principle with rate function given by~\eqref{eq:ratefunction_main} will follow from \cite[Theorem 1]{biggins}. The details of this construction is given below.

\begin{proof}[Proof of Theorem~\ref{theorem:main}]

Let $\Lambda = \mathcal{P}(\Xi^{2}_{\leq}) \times \mathcal{P}(\Theta)$ and
\begin{equation}
\Lambda_{n} = \Big\{ \Big( \tfrac{\vec{m}_{n,\leq}}{ \|\vec m_n\|_{1} }, \tfrac{\vec{u}_n}{n} \Big) : \|\vec m_n\|_{1} = \tfrac{1}{2} \sum_{i=1}^{n}\ell_{n,i} \text{ and } \|\vec u_n\|_{1} = n \Big\} \subset \Lambda.
\end{equation}
Recall from~\eqref{eq:b} that $\frac{1}{2n}\sum_{i=1}^{n} \ell_{n,i} \to \frac{d}{2}$.

Let $\mu_{n}$ denote the distribution of $\big( L_{m_n}(\vec X_\leq), L_n(\vec O) \big)$ defined in $\Lambda_{n}$ with independent coordinates.
It follows from Sanov's Theorem that $(\mu_n)_{n \in \N}$ is exponentially tight and satisfies a large deviation principle with rate function
\begin{equation}
\psi(\vec{\alpha}, \vec{\beta}) = \frac{d}{2}H( \vec{\alpha} | \chi^{2}_\leq) + H( \vec{\beta} | \vartheta).
\end{equation}

For each $ \gamma := \Big( \tfrac{\vec{m}_{n,\leq}}{\|\vec m_n\|_{1}}, \tfrac{\vec{u}_n}{n} \Big)  \in \Lambda_n$  let  $\P_{\gamma}$  be the uniform distribution on $\bar{\mathcal{G}}(\vec{\ell}_n)_{\vec{m}_n,\vec{u}_n}$. Since $\Lambda_n$ is finite, for any event $A \subset \bar{\mathcal{G}}(\vec{\ell}_n)_{\vec{m}_n,\vec{u}_n}$,  the function $\gamma \mapsto \P_{\gamma}(A)$ is measurable.

Assume now that $(\vec{m}_{n,\leq}, \vec{u}_n)$ is such that
\begin{equation}
     \Big( \tfrac{\vec{m}_{n,\leq}}{ \|\vec m_n\|_{1} }, \tfrac{\vec{u}_n}{n} \Big) \to (\vec{\alpha}, \vec{\beta}).
\end{equation}
In this case, $(\vec{m}_{n}, \vec{u}_n)$ is adapted to $(d\vec{\alpha}_{+}, \vec{\beta})$ (in the sense of Definition~\ref{def:adapted}), where
\begin{equation}
\alpha_{+}^{x,y} = \begin{cases}
\frac{\alpha^{x,y}}{2}, & \quad \text{if } x < y, \\
\frac{\alpha^{y,x}}{2}, & \quad \text{if } y < x, \\
\alpha^{x,y}, & \quad \text{if } x=y.
\end{cases}
\end{equation}
From Theorem \ref{ldp:uniform_given_neihborhood}, if $\bar{H}_n$ denoted the uniformly sampled graph in $\bar{\mathcal{G}}(\vec{\ell}_n)_{ \vec{m}_n,\vec{u}_n}$, then the sequence $\big( U(\bar{H}_n) \big)_{n \in \N}$ satisfies a large deviation principle with rate function $I_{P, d \vec{\alpha}_{+}, \vec{\beta}}$ given by~\eqref{eq:ratefunction_l:6.1}.

The facts highlighted above, when combined with Biggins Theorem~\cite[Theorem 1]{biggins}, implies that $(U(\bar{G}_n))_{n\in\mathbb{N}}$ satisfies a large deviation principle with speed $n$ and rate function defined by, for any $\mu \in \mathcal{P}(\bar G_*)$, 
\begin{equation*}
\lambda_{P}(\mu) = \inf\bigg\{ I_{P, d \vec{\alpha}_{+}, \vec{\beta}}(\mu) + \frac{d}{2} H( \vec{\alpha} | \chi^{2}_\leq) + H( \vec{\beta} |\nu) : ( \vec{\alpha} , \vec{\beta} ) \in \mathcal{P}(\Xi^2_\leq) \times \mathcal{P}(\Theta) \bigg\}.
\end{equation*}
It remains to verify that the expression above coincides with~\eqref{eq:ratefunction_main}. Notice first that, due to the expression of $I_{P, d \vec{\alpha}, \vec{\beta}}$ in~\eqref{eq:ratefunction_l:6.1}, in order for this quantity to be finite, it is necessary that $d \vec{\alpha}_{+} = \vec{\deg}(\mu)$, $\vec{\beta} = \vec{\Pi}(\mu)$ (see the paragraph immediately after Theorem~\ref{t:equality_entropy}). In particular, this implies $\deg(\mu) = d$ and $\vec{\alpha} = \frac{1}{d} \vec{\deg}(\mu)_{\leq}$. This concludes the proof of the theorem.
\end{proof}

\subsection{Colored Configuration Model}
\label{sec:color-conf-model}
~

In this section, we review the main properties of the colored configuration model, first introduced in \cite[Section~4]{bc}. This model and some of its properties will be necessary in the proof of Lemma~\ref{lemma:modification} in Section~\ref{subsec:proof_lemma_modification}.

Let $L \ge 1$ be a fixed integer, and define $\mathcal{C} = \{ (i,j) : 1 \leq i , j \leq L \}$. Let furthermore $\mathcal{C}_{=} = \{(i,i) \in \mathcal{C}\}$, $\mathcal{C}_{<} = \{(i,j) \in \mathcal{C}: i<j\}$, and $\mathcal{C}_{\leq} = \mathcal{C}_{=} \cup \mathcal{C}_{<}$.

Each element $(i,j) \in \mathcal{C}$ is interpreted as a color. For $c := (i,j) \in \mathcal{C}$, we use the notation $\bar{c} := (j,i)$ for the conjugate color. Notice that $\bar{c} = c$ for $c \in \mathcal{C}_{=}$.

Define now the set $\widehat{\mathcal{G}}(\mathcal{C})$ of directed colored multigraphs with colors in $\mathcal{C}$, comprised of locally-finite multigraphs\footnote{A multigraph allows for multiple edges and self-loops.} with (oriented) edges colored with elements in $\mathcal{C}$ in a consistent way. More precisely, each $G \in \widehat{\mathcal{G}}(\mathcal{C})$ is of the form  $G = (V, \omega)$ where $V$ is denotes the vertex set, and $\omega = (\omega_c: c \in \mathcal{C})$ with $\omega_{c}: V^{2} \to \Z_{+}$ for each $c \in \mathcal{C}$, satisfies the following.
\begin{enumerate}
\item For all $u \in V$ and $c \in \mathcal{C}$, $\sum_{v \in V} \omega_c(u,v) < \infty$.

\item For $c \in \mathcal{C}$, $\omega_c(u,v) = \omega_{\bar{c}}(v,u)$ for all $u,v \in V$.

\item For $c \in \mathcal{C}_=$, $\omega_c(u,u)$ is even for all $u \in V$.
\end{enumerate}

For $G = (V, \omega) \in \widehat{\mathcal{G}}(\mathcal{C})$ the associated colorblind multigraph is given by ${\rm CB}(G)= (V, \bar{\omega})$, with $\bar{\omega}: V^2 \rightarrow \mathbb{Z}_+$ defined as
\begin{equation}
  \bar{\omega}(u,v) = \sum_{c \in \mathcal{C}} \omega_c(u,v).
\end{equation}
Notice that distinct colored multigraphs can be associated to the same colorblind multigraphs. 

A colored graph is a multigraph $G \in \widehat{\mathcal{G}}(\mathcal{C})$, such that ${\rm CB}(G)$ has no multiple edges nor self-loops. Denote by $\mathcal{G}(\mathcal{C}) \subset \widehat{\mathcal{G}}(\mathcal{C})$ the subset of colored graphs.

Let $\mathcal{M}_L$ denote the set of $L$ by $L$ matrices with nonnegative integer valued entries. For a finite colored multigraph $G = (V, \omega) \in \widehat{\mathcal{G}}(\mathcal{C})$, $u \in V$, and $c \in \mathcal{C}$, define
\begin{equation} \label{eq:coloredDegree}
  D^G_c(u) = \sum_{v \in V} \omega_c(u,v),
\end{equation}
the number of edges of color $c$ going out of $u$. Let $D^{G}(v) = (D^{G}_{c}(v): c \in \mathcal{C})$, and notice that $D^{G}(v) \in \mathcal{M}_L$. This matrix is called the colored-degree matrix of the vertex $v$. We call $\vec{D}^G = (D^G(v): v \in V)$ the colored-degree sequence corresponding to $G$.

\bigskip

We now introduce the colored configuration model. For $n \in \mathbb{N}$, let $\mathcal{D}_n$ denote the set of vectors $(D(1), \dots, D(n))$ where, for each $1 \leq i \leq n$, $D(i) = (D_{c}(i): c \in \mathcal{C}) \in \mathcal{M}_L$ and $S = \sum_{i=1}^{n} D(i)$ is a symmetric matrix with even coefficients on the diagonal (in the sense that $S_{c} = S_{\bar{c}}$ and,  for every $c \in \mathcal{C}_{=}$, $S_{c}$ is even). Notice that, for $G \in \widehat{\mathcal{G}}(\mathcal{C})$, $\vec{D}^G \in \mathcal{D}_n$.
 
Given a vector $\vec{D} \in \mathcal{D}_n$, consider the set $\widehat{\mathcal{G}}(\vec{D})$ of directed colored multigraphs $G \in \widehat{\mathcal{G}}(\mathcal{C})$ with $n$ vertices such that $\vec{D}^{G} = \vec{D}$. Furthermore, for $h \geq 1$, let $\widehat{\mathcal{G}}(\vec{D}, h) \subset \widehat{\mathcal{G}}(\vec{D})$ be the subset of directed colored multigraphs $G$ such that ${\rm CB}(G)$ has no cycles of length at most $h$. Note that $\mathcal{G}(\vec{D}, h+1) \subseteq \mathcal{G}(\vec{D}, h)$ and that $\mathcal{G}(\vec{D}, 2) \subset \mathcal{G}(\mathcal{C})$. In particular, it holds that $\mathcal{G}(\vec{D},h) \subset \mathcal{G}(\mathcal{C})$, for all $h \geq 2$.

Now, given $\vec{D} = (D(1), \dots, D(n)) \in \mathcal{D}_n$, we refer to the colored configuration model as the directed colored multigraph chosen uniformly at random from the set $\widehat{\mathcal{G}}(\vec{D})$. Let us now present an algorithmic procedure that generates such random graphs. This procedure is similar to the one used for the classical configuration model.

For each $c \in \mathcal{C}$, let $W_c := \cup_{i=1}^n W_c(i)$ be a set of distinct half edges of color $c$, where $|W_c(i)| = D_c(i)$, for each $i \leq n$. Attach the half edges in $W_c(i)$ to vertex $i$. A half edge with color $c$ is required to be connected to another half edge with color $\bar{c}$. In order to achieve this, consider, for $c \in \mathcal{C}_{<}$, the set $\Sigma_c$ of bijections $\sigma_{c}: W_{c} \rightarrow W_{\bar{c}}$ (notice this set is nonempty since $\vec{D} \in \mathcal{D}_{n}$, and thus $|W_{c}| = |W_{\bar{c}}|$). Likewise, for $c \in \mathcal{C}_{=}$, let $\Sigma_c$ denote the set of perfect matchings of the set $W_{c}$, which is nonempty since $\vec{D} \in \mathcal{D}_{n}$ implies that $|W_c|$ is even, for all $c \in \mathcal{C}_{=}$.

Let now $\Sigma = \prod_{c \in \mathcal{C}_{\leq}} \Sigma_{c}$ and denote by $\sigma = (\sigma_{c} \in \Sigma_{c}: c \in \mathcal{C}_\leq) \in \Sigma$ for an element of the set $\Sigma$. Given $\sigma \in \Sigma$, construct a directed colored multigraph $\Gamma(\sigma) \in \widehat{\mathcal{G}}(\vec{D})$ in the following way. For $c \in \mathcal{C}_{\leq}$, use the map $\sigma_c$ to construct directed edges between half edges of color $c$ and $\bar{c}$ and vice versa, that is, if $\sigma_{c}$ maps a half edge of color $c$ at vertex $u$ to a half edge of color $\bar{c}$ at
vertex $v$, place an edge directed from $u$ towards $v$ with color $c$ and an edge directed from $v$ towards $u$, having color $\bar{c}$. In the case of $c \in \mathcal{C}_{=}$, this construction adds two directed edges, one in each direction from vertices $u$ and $v$ with the same color $c$ (notice that $\bar{c}=c$). In all the algorithm, we also allow for $u=v$. By construction, we have $\Gamma(\sigma) \in \widehat{G}(\mathcal{C})$, for each $\sigma \in \Sigma$. For $\vec{D} \in \mathcal{D}_n$, let ${\rm CM}(\vec{D})$ denote the random directed multigraph $\Gamma(\sigma)$ obtained by choosing $\sigma \in \Sigma$ uniformly at random.

\bigskip

We end this section with a central result first proved in~\cite{bc}. Given a positive integer $\delta$, let $\mathcal{M}_L^{(\delta)} \subset \mathcal{M}_{L}$ denote the set of $L$ by $L$ matrices with non negative integer entries bounded by $\delta$. Fix $R \in \mathcal{P}(\mathcal{M}_L^{(\delta)})$ and let $\vec{D}_{n} = ( D_{n}(1), \dots, D_{n}(n)) \in \mathcal{D}_n$ be a sequence satisfying
\begin{subequations}
  \begin{gather}
    D_{n}(i) \in \mathcal{M}_L^{(\delta)}, \qquad \text{for all } i  \leq n, \label{eq:H1} \\
    \frac{1}{n} \sum_{i=1}^n \delta_{D_{n}(i)} \rightarrow R. \label{eq:H2}
  \end{gather}
\end{subequations}

\begin{theorem}[\cite{bc}, Theorem 4.5]\label{thm:alpha-h}
Fix $\delta \in \mathbb{N}$, $R \in \mathcal{P}(\mathcal{M}_L^{(\delta)})$, a sequence $\vec{D}^{(n)}$ sa\-tisfying~\eqref{eq:H1} and~\eqref{eq:H2}, and let $G_n \sim {\rm CM}(\vec{D}_{n})$. Then, for every $h \geq 1$, there exists $\alpha_h > 0$ such that
  \begin{equation*}
    \lim_{n \rightarrow \infty} \P ( G_n \in \mathcal{G}(\vec{D}_{n}, h)) = \alpha_h.
  \end{equation*}
In particular, for any $h \geq 1$, the set $\mathcal{G}(\vec{D}_{n}, h)$ is non empty for all $n$ large enough.
\end{theorem}

\subsection{Proof of Lemma~\ref{lemma:modification}}\label{subsec:proof_lemma_modification}
~

In  Lemma~\ref{lemma:modification} we start with a sequence of probability distributions on trees with height one $P_{n}$ and degree bounded by $L$ converging to a limiting distribution $P$ and $\rho \in \mathcal{P}(\bar{\mathcal{T}}_{*}^{k})$ supported on marked trees with degrees bounded by $L>0$ (and thus with finite support) and such that $\rho_{1} \circ \pi_{g}^{-1} =P$. Our goal is to construct a sequence of marked graphs $(\tilde{\Gamma}_{n})_{n \in \N}$ such that $U(\tilde{\Gamma}_{n})_{k} \to \rho_{k}$ and $U(\pi_{g}(\tilde{\Gamma}_{n}))_{1} = P_{n}$.

In view of~\cite[Lemma~6]{da}, there exists a finite set $\Delta_k \subset \bar{\mathcal{T}}_{*}^{k}$ and a sequence $\{\Gamma_{n}\}_{n \in \N}$ for which the support of $U(\Gamma_{n})_{k}$ is contained in $\Delta_k$ and $U(\Gamma_{n})_{k} \rightarrow \rho_k$. In order to conclude the proof, it suffices to prove that there exists, for every $n$ large enough, a modification $\widetilde{\Gamma}_{n}$ of the graph $\Gamma_{n}$ in a way that
\begin{enumerate}

\item $[\widetilde{\Gamma}_{n}, i]_{k} = [\Gamma_{n}, i]_{k}$, for all but $o(n)$ vertices $i$, up to a permutation of the vertices,

\item $U(\pi_{g}(\widetilde{\Gamma}_{n}))_1 = P_{n}$, for all $n$ large enough.

\end{enumerate}

The colored configuration model will be used to construct the modification above, but we first need to introduce some additional notation. For a marked graph $\bar{G}$ and vertices $u \sim v$, set 
\begin{equation*}
\bar{G}(u,v)=(\xi_{\bar{G}} (u,v),(\bar{G}' ,v)),
\end{equation*}
where $\bar{G}'$ is the connected component of $v$ in the graph obtained from $\bar{G}$ by removing the edge between $u$ and $v$. For an integer $k \geq 1$, $\bar{G}(u,v)_k$ is defined as $(\xi_{\bar{G}} (u,v),(\bar{G}' ,v)_k)$. Let $\bar{G}[u,v]$ denote the pair $(\xi_{\bar{G}} (u,v),[\bar{G}' ,v]) \in \Xi \times
\bar \sG_*$ and, for $k \geq 1$, let $\bar{G}[u,v]_k$ denote $(\xi_{\bar{G}} (u,v),[\bar{G}' ,v]_k ) \in \Xi \times \bar{\sG}^k_*$.
\begin{definition}
Let $S \subset \bar{\mathcal{G}}_*^{k}$ be a finite set of marked rooted trees. Consider $\mathcal{F} \subset \Xi \times \bar{\mathcal{G}}_*^{k-1}$ the set comprised of $G [o, v]_{k-1}$ and $G [v,o]_{k-1}$ for each $[G, o] \in S$ and $v \sim_G o$. Since $S$ is finite, $\mathcal{F}$ is finite. We define the set of colors coming from $S$ to be $\mathcal{C} = \mathcal{F} \times \mathcal{F}$.
\end{definition}

For a marked graph $G$ with $n$ vertices and any integer $k \geq 1$, we define a directed colored graph denoted $C(G)$ in the following way. Let $\mathcal{C}$ denote the colors coming from $S = \{[G(u)_{k},u]: u \in [n] \}$ as in the definition above. For any two adjacent vertices $u \sim_{G} v$, we add a directed edge in $C(G)$ from $u$ to $v$ with color $( G[u,v]_{k-1}, G[v,u]_{k-1})$ and another directed edge from $v$ towards $u$ with the conjugate color. Since $G$ is simple, $C(G)$ is a directed colored graph, that is, $C(G) \in \mathcal{G}(\mathcal{C})$.

On the other hand, for any fixed integer $k \geq 1$ and finite set $\mathcal{F} \subset \Xi \times \bar{\mathcal{G}}_{*}^{k-1}$, let $\mathcal{C} = \mathcal{F} \times \mathcal{F}$. Given a directed colored graph $H \in \mathcal{G}(\mathcal{C})$ on $n$ vertices and a sequence or vertex marks  $\vec{\tau} = (\tau(v) : v \in [n] )$, the \emph{marked color blind} version of $(\vec{\tau} ,H)$, denoted by $\text{MCB}_{\vec{\tau}}(H)$ is defined as follows. Each vertex $v \in [n]$ receives the mark $\vec{\tau}(v)$. Furthermore, for adjacent vertices $u \sim_{H} v$ with color of the edge directed from $u$ to $v$ given as $((\xi, t), (\xi', t')) \in \mathcal{C}$, we put an edge between $u$ and $v$ in $\text{MCB}_{\vec{\tau}} (H)$ with mark $\xi$ directed from $u$ towards $v$ and $\xi'$ directed from $v$ towards $u$.

Notice however that it is not entirely clear if the colors of a marked graph $H$ are consistent with the colors of $C(\text{MCB}_{\vec{\tau}}(H))$. This is the content of the next proposition, stated in~\cite{da} as Corollary 3.

\begin{proposition}[\cite{da}, Corollary 3]
For an integer $k \geq 1$, assume that a marked $k$-tree–like graph $G \in \bar{\mathcal{G}}_{n}$ is given. Let $\vec{D} = \vec{D}^{G}$ be the colored degree sequence associated to $C(G)$. Let $\vec{\tau} = (\tau(v) : v \in [n])$ denote the vertex mark vector of $G$. Then, for
any directed colored graph $H \in \mathcal{G} ( \vec{D}, 2k+1)$, it holds that $(\textnormal{MCB}_{\vec{\tau}}(H),  v)_{k} = (G, v)_{k}$, for all $ v \in [n]$.
\end{proposition}

Recall that $U(\Gamma_{n})_k$ is supported on the finite set $\Delta_{k}$ and let $\mathcal{C}_k$ be the colors coming from $\Delta_{k} \subset \bar{\mathcal{T}}^{k}_{*}$. We denote by $D_n=\{D_{n}(c)\}_{c \in \mathcal{C}_{k}}$ the color-degree sequence associated to $\Gamma_n$ and observe that $\sum_{c} D_n(c)(i)=\deg_{\Gamma_n}(i)$, for each vertex $i$. Since $P_{n} = \frac{1}{n}\sum_{i=1}^{n}\delta_{\ell_{n,i}} \rightarrow P$ and $U(\pi_{g}(\Gamma_n))_1 \rightarrow \rho_1 \circ \pi_{g}^{-1} = P$, it follows that $\deg_{\Gamma_n}(i) = \ell_{n,i}$ for all but $o(n)$ vertices, up to permutation of the vertices.

\begin{lemma}\label{le:matrix} There exists $\widetilde D_n\in \mathcal{D}_n$ satisfying the following.
\begin{enumerate}
    \item The vector $\widetilde D_n$ is a modification of $D_n$ in the sense that there exists at most $o(n)$ values of $i \in [n]$ such that $\widetilde D_n(i) \neq D_n(i)$.
    
    \item $\sum_c\widetilde D_n(c)(i)=\ell_{n,i},$ for all $i\in [n]$.
\end{enumerate}
\end{lemma}

The proof of this lemma is postponed to Section~\ref{sec:matrix} and first use it to verify Lemma~\ref{lemma:modification}.
\begin{proof}[Proof of Lemma~\ref{lemma:modification}]
Lemma~\ref{le:matrix} implies that there exists $\widetilde{D}_n \in \mathcal{D}_{n}$ which coincides with $D_n$ in all but at most $o(n)$ vertices  (recall the definition of $\mathcal{D}_{n}$ in Section~\ref{sec:color-conf-model}) satisfying
\begin{equation}\label{eq:h=1compatibility}
\sum_c \widetilde{D}_n(c)(i)=\ell_{n,i}.
\end{equation}
Theorem~\ref{thm:alpha-h} now implies that, for $n$ large enough, there exists $H_n \in \mathcal{G}(\widetilde{D}_n, 2k+1)$. We then define $\widetilde{\Gamma}_n={\rm MCB}_{\vec{\tau}_{\Gamma_n}}(H_n)$ which satisfies
\begin{enumerate}
\item From \eqref{eq:h=1compatibility}, $\deg_{\widetilde{\Gamma}_n}(i)=\ell_{n,i}$ for all vertices up to a vertex permutation.

\item By construction of $\widetilde{D}_n$, $[\widetilde{\Gamma}_n,i]_k=[\Gamma_n,i]_{k}$ for all but $o(n)$ vertices.
\end{enumerate}

Therefore, $\widetilde{\Gamma}_n$ satisfies the required conditions of the lemma and concludes the proof.
\end{proof}

\section{``Mass transport'' on matrices of colored degree}\label{sec:matrix}

This section contains the proof of Lemma~\ref{le:matrix}. We first introduce the set of matrices $\mathfrak{D}_{n}$ that mimics the colored-degree matrices in the set $\mathcal{D}_{n}$.
\begin{definition}
Let $\mathfrak{D}_{n}$ denote the set of $(p+2m) \times n$ matrices $A = (a_{i,j})$ with non-negative integer entries such that
\begin{equation}\label{eq:off_diagonal}
\sum_{j=1}^{n} a_{p+i,j} = \sum_{j=1}^{n} a_{p+i+1,j}, \quad \text{for all } 1 \leq i \leq 2m \text{ odd},
\end{equation}
and
\begin{equation}\label{eq:diagonal}
\sum_{j=1}^{n} a_{i,j}  \text{ is even}, \quad \text{for all } 1 \leq i \leq p.
\end{equation}
\end{definition}

For $A \in \mathfrak{D}_{n}$, let
\begin{equation}
\alpha_j = \sum_{i=1}^{2m+p} a_{i,j}
\end{equation}
denote the sum of the entries in column $j$ and define
\begin{equation}
    \deg A =(\alpha_1, \dots, \alpha_n).
\end{equation}

The next result is the main step towards the proof of Lemma~\ref{le:matrix}.
\begin{lemma}\label{lemma:matrix_degree}
Let $A_{n} = (a^{(n)}_{i,j}) \in \mathfrak{D}_{n}$ with entries bounded by $L \geq 0$ and assume that $\vec{\beta}_{n} = (\beta_{n,j})_{1 \leq j \leq n}$ is a sequence of vectors with non-negative integer entries bounded by $M \geq 0$ with $\sum_{j=1}^{n} \beta_{n,j}$ even and such that  $\vec{\beta}_{n}$ coincides with $\deg A_{n}$ in all but $o(n)$ entries. Then, there exists $A_{n}' \in \mathfrak{D}_{n}$ such that $\deg A_{n}' = \vec{\beta}_{n}$ and with entries bounded by $M$.
\end{lemma}

\begin{proof}
This proof will be divided in three cases, according to the values of $m$ and $p$. We first treat the case $m=0$ and $p=1$, followed by the case $m=1$ and $p=0$, and conclude with the general case.

We denote by $\deg A_{n} = (\alpha^{(n)}_{1}, \dots, \alpha^{(n)}_{n})$ and let $s_{n} = o(n)$ denote the cardinality of the set $I_{n} = \{ j \leq n: \beta_{n,j} \neq \alpha^{(n)}_{j} \}$.

\bigskip

\noindent{\textbf{The case $p=1$ and $m=0$.}}
In this case the matrix $A_{n}$ coincides with $\deg A_{n}$ and the modification is straightforward: simply define $A_{n}' = \vec{\beta}_{n}$. By assumption, $\alpha^{(n)}_j \neq \beta_{n,j}$ if, and only if, $j \in I_{n}$, which implies that only $s_{n} = o(n)$ columns of $A_{n}$ were modified.

\bigskip

\noindent{\textbf{The case $p=0$ and $m=1$.}} Let us assume without loss of generality that $1 \notin I_{n}$. The construction of $A_{n}'$ is made in two steps.

\medskip

\noindent\textit{Step 1.} The first step consists in moving all the excess to the first column. Given the matrix $A_{n}$, we construct the auxiliary matrix $B_{n} = (b^{(n)}_{i,j})$ via
\begin{equation}
b^{(n)}_{i,j} = \begin{cases}
    a^{(n)}_{i,j}, & \quad \text{if } j \notin I_{n} \cup \{1 \}, \\
    \beta_{n,j}, & \quad \text{if } i=1 \text{ and } j \in I_{n}, \\
    0, & \quad \text{if } i=2 \text{ and } j \in I_{n}, \\
    \beta_{n,1} + \sum_{k \in I_{n} \cup \{1\}} a^{(n)}_{1,k}, & \quad \text{if } i=j=1,  \\
    \sum_{k \in I_{n} \cup \{1\}} a^{(n)}_{2,k} + \beta_{n,k}, & \quad \text{if } i=2, j=1.
\end{cases}
\end{equation}

The definition of the entries $b^{(n)}_{1,1}$ and $b^{(n)}_{2,1}$ is done just so that the matrix $B_{n}$ satisfies~\eqref{eq:off_diagonal}. Furthermore, it is clear that the entries of the matrix $B_{n}$ are non-negative integer and that $\deg B_{n} = (\beta_{n,1} + r_{n}, \beta_{n,2}, \dots, \beta_{n,n})$, where $r_{n}$ is a positive even integer value bounded by $\sum_{j \in I_{n} \cup \{1\}} a^{(n)}_{1,k} + a^{(n)}_{2,k} + \beta_{n,j} \leq (2L+M)(s_{n}+1) = o(n)$.

\medskip

\noindent\textit{Step 2.} We now construct matrices $A_{n}'$ that take care of these additional factors of $r_{n}$ in the sum of the first column.

If $b^{(n)}_{1, 1} = b^{(n)}_{2, 1}$, then we obtain $A_{n}'$ from $B_{n}$ by simply removing $\frac{r_{n}}{2}$ from each entry in the first column (notice that $2b^{(n)}_{1,1} = 2b^{(n)}_{2, 1} = b^{(n)}_{1,1}+b^{(n)}_{2, 1} = \beta_{n,1}+r_{n} \geq r_{n}$). Suppose then without loss of generality that $b^{(n)}_{1, 1} > b^{(n)}_{2, 1}$, because the other case is treated analogously.

Notice that, since  the matrix $B_{n}$ satisfies~\eqref{eq:off_diagonal} and $r_{n}>0$, there exists $k > 1$ such that $b^{(n)}_{2,k} \geq 1$. We now reduce the value of entry $b^{(n)}_{1,1}$ by 2, increase the entry $b^{(n)}_{1, k}$ by 1, and reduce the value $b^{(n)}_{k,2}$ by 1. This process is then repeated a maximum of $\frac{r_{n}}{2}$ times in order to obtain a matrix $A_{n}'$ with non-negative integer entries that still satisfies~\eqref{eq:off_diagonal} and such that $\deg A_{n}' = \vec{\beta}_{n}$. Finally, notice that at most $s_{n}+r_{n}+1 = o(n)$ columns where modified to go from $A_{n}$ to $A_{n}'$.

To conclude the proof of this case, it remains to notice that the entries of $A_{n}'$ are bounded by the maximum value of the vector $\vec{\beta}_{n}$.

\bigskip

\noindent{\textbf{The general case.}} In the general case we have $m$ pairs of conjugated colors and $p$ diagonal colors. Recall that we already know how to solve the cases $(m,p)=(1,0)$ and $(m,p)=(0,1)$. We now argue how to solve the problem for any $(m,p)\in \bb N\times \bb N$. Let $A_{n}$ be the matrices associated to the problem of type $(m,p)$. As before, let $I_{n}$ denote the set of columns $j \leq n$ such that $\beta_{n,j} \neq \alpha^{(n)}_{j}$.

We split the matrix $A_{n}$ in several submatrices $A^{(i)}_{n}$ for $i \leq p+m$ such that each submatrix $A^{(i)}_{n}$ corresponds to either a diagonal color or to a pair of conjugated colors: if $i \leq p$, then $A^{(i)}_{n}$ is simply the $i$th row of the matrix $A_{n}$. If on the other hand $p < i \leq p+m$, $A^{(i)}_{n}$ corresponds to the matrix with two rows given by selecting rows $p+2i-1$ and $p+2i$ of $A_{n}$.

The second step is to define the target degrees $\vec{\beta}^{(i)}_{n}$. The case $i=1$ is special and will be treated separately. Consider for now $i \geq 2$ and let us focus in the case $2 < i \leq p$, since $i > p$ is treated analogously. If $\sum_{j \notin I_{n}} a^{(n)}_{i, j}$ is even, we simply set $\beta^{(i)}_{n,j} := a^{(n)}_{i, j} \textbf{1}_{j \notin I_{n}}$. If $\sum_{j \notin I_{n}} a^{(n)}_{i,j}$ is odd, select $j_{*} \notin I_{n}$ such that $a^{(n)}_{i, j_{*}} >0$ and set $\beta^{(i)}_{n,j} := a^{(n)}_{i, j} \textbf{1}_{j \notin I_{n} \cup \{j_{*}\}} + (a^{(n)}_{i, j_{*}}-1) \textbf{1}_{j_{*}}$. With these definitions it follows that $\sum_{j=1}^{n}\beta^{(i)}_{n,j}$ is even for each $2\leq i\leq p+m$. Notice that each of these submatrices have at most $s_{n}+1$ columns such that the target degree does not coincide with the total degree. Finally, we define $\beta^{(1)}_{n,j} := \beta_{n,j} - \sum_{i=2}^{p+m}\beta^{(i)}_{n,j}$. It follows straight from the definition that $\vec{\beta}^{(1)}_{n}$ has only non-negative integer entries (since $\beta^{(i)}_{n,j} = 0$ for $i \geq 2 $ and $j \in I_{n}$ and $\sum_{i=1}^{p+2m}a^{(n)}_{i,j} = \beta_{n,j}$, for $j \notin I_{n}$) such that $\sum_{j=1}^{n} \beta^{(1)}_{n,j}$ is even. Furthermore, $\vec{\beta}^{(1)}_{n}$ disagrees with $ \deg A^{(1)}_{n}$ in at most $p+m-1 = o(n)$ entries. 

We now apply the methods of the previous cases to solve each subproblem and obtain $(A^{(i)}_{n})'$, for $1 \leq i \leq p+m$. This concludes the proof.
\end{proof}

We are now ready to present the proof of Lemma~\ref{le:matrix}.
\begin{proof}[Proof of Lemma~\ref{le:matrix}]
    The proof follows from a direct application of Lemma~\ref{lemma:matrix_degree}. Let $p$ denote the number of diagonal colors in $\mathcal{C}_k$ while $m$ counts pairs of conjugated colors. For each color $c \in \mathcal{C}_k$ we associate an index from the set $\{1, \dots, 2m+p\}$ such that the colors $\{1, \dots, p\}$ are the diagonal colors and the indices $(p+i, p+i+1)$, for $1 \leq i \leq 2m$ odd, correspond to pairs of conjugated colors. With the above conventions we write
    \begin{enumerate}
        \item $a_{ij}=D_n(c)(j)$, $1\leq j\leq n$ where $c$ is the $i$th-diagonal color, and
        \item $a_{p+i,j}=D_n(c)(j)$, $a_{p+i+1,j}=D_n({\bar c})(j)$ where $(c,\bar c)$ is the pair of conjugated  colors corresponding to the indices $(p+i,p+i+1)$, for $1\leq i\leq 2m$ odd.
    \end{enumerate}
   Set $\beta_{n,j} = \ell_{n,j}$, for $j \in [n]$. Lemma~\ref{lemma:matrix_degree} yields a modification $A_{n}'$ of the matrix $A_{n}$ such that 
    \begin{equation}
        \deg A_{n}'= \vec{\beta}_{n}.
    \end{equation}
    This concludes the proof.
\end{proof}

\bibliographystyle{plain} 

\bibliography{mybib}

\end{document}